\documentclass[11pt]{amsart}

\usepackage{amssymb,latexsym,amsmath,amsthm,caption}
\usepackage{graphicx,multicol,enumitem}
\usepackage{fullpage}

\newcounter{Theorem}
\numberwithin{Theorem}{section}
\newtheorem{theorem}[Theorem]{Theorem}

\newtheorem{proposition}[Theorem]{Proposition}

\newtheorem*{theorem*}{Theorem}
\newtheorem*{proposition*}{Proposition}
\newtheorem*{lemma*}{Lemma}
\theoremstyle{definition}
\newtheorem*{remark*}{Remark}
\newtheorem{example}{Example}

\DeclareMathOperator*{\argmin}{arg\,min}

\title{A Greedy Version of the Frame Algorithm}

\author{Brody Dylan Johnson}
\email{brody.johnson@slu.edu}
\address{Department of Mathematics and Statistics, Saint Louis University, St. Louis, Missouri 63103, USA}
\date{\today}

\begin{document}

\begin{abstract}
The frame algorithm uses a simple recursive formula to approximate an unknown vector from its frame coefficients.  This note introduces an adaptive version of the frame algorithm that maximizes the error reduction between steps in terms of an equivalent norm.  This greedy version of the frame algorithm is proven to achieve the same guaranteed convergence rate as the traditional frame algorithm, yet, unlike its classical counterpart, does not require knowledge of the frame bounds.  The robustness of the greedy frame algorithm with respect to noisy measurements is also established.  Two numerical examples are included to demonstrate the benefit of the greedy algorithm in applications.
\end{abstract}

\keywords{frame, frame algorithm, greedy algorithm}
\subjclass{42C15, 42A65, 65J10}

\maketitle

\section{Introduction} \label{section:1}

Recall that a collection $\lbrace x_{j} \rbrace_{j\in J}$ is a \emph{frame} for a separable Hilbert space $\mathbb{H}$ if there exist positive constants $A$ and $B$ such that 
$$ A \Vert x\Vert^{2} \le \sum_{j\in J} \vert \langle x, x_{j} \rangle \vert^{2} \le B \Vert x\Vert^{2} \qquad \text{for all} \quad x\in \mathbb{H}. $$

\noindent
The constants $A$ and $B$, respectively, are referred to as the lower and upper frame bounds.  A fundamental problem in frame theory, referred to here as the \emph{frame inversion problem}, regards the reconstruction of a vector $x$ from its frame coefficients, $\lbrace \langle x, x_{j} \rangle \rbrace_{j\in J}$.  Often, the frame coefficients may be regarded as measurements of an unknown signal $x$, leading naturally to related problems in which an unknown signal must be approximated from corrupted measurements.  Duffin and Schaeffer provided one solution of the frame inversion problem via successive approximation in the same paper in which they introduced frame theory to the mathematical literature \cite[Theorem III]{DuffinSchaeffer1952}.  A modest generalization of Duffin and Schaeffer's original recursive scheme, including a relaxation parameter $\alpha$, is stated below.

\begin{theorem*}[Frame Algorithm \cite{Grochenig1993}]
Let $\lbrace x_{j} \rbrace_{j\in J}$ be a frame for a separable Hilbert space $\mathbb{H}$ with bounds $0<A\le B<\infty$.  Fix $x\in \mathbb{H}$.  Let $y_{0}=0$ and, for $n\ge 0$, define
\begin{equation} \label{eq:frame-algorithm}
y_{n+1} = y_{n} + \alpha \sum_{j\in J} \langle x-y_{n}, x_{j} \rangle x_{j}. 
\end{equation}

\noindent
If $0<\alpha < \frac{2}{B}$, it follows that $\Vert x- y_{n}\Vert \le C_{\alpha}^{n} \Vert x \Vert$, where $C_{\alpha}: = \max{\lbrace \vert 1-\alpha A\vert, \vert 1-\alpha B\vert \rbrace}$. In particular, the optimal constant $C_{\alpha}$ is $\frac{B-A}{B+A}$ and occurs when $\alpha = \frac{2}{A+B}$.
\end{theorem*}

The frame algorithm should be regarded a special case of Richardson's method \cite{DMYoung1980} in which the convergence criterion and error estimate stem from functional analytic properties of the associated frame.  In real world applications, the frame algorithm is susceptible to several potential hazards, chief among them being the fact that selection of a near-optimal relaxation parameter relies on precise estimates of both the lower and upper frame bounds.  Another shortcoming of the frame algorithm arises when $B/A>>1$, which results in a poor convergence rate even when the optimal frame bounds are known.  Gr\"ochenig addressed these concerns by leveraging polynomial acceleration methods from numerical linear algebra in the creation of two reconstruction algorithms based on two-step recurrence formulas \cite{Grochenig1993}.  The Chebyshev algorithm \cite[Theorem 1]{Grochenig1993} uses a static recurrence relation to improve the rate of convergence given knowledge of the frame bounds, while the conjugate-gradient algorithm \cite[Theorem 2]{Grochenig1993} implements an adaptive recurrence relation that provides an improved rate of convergence in terms of an equivalent norm without requiring knowledge of the frame bounds.

The goal of this work is to introduce an adaptive, single-step version of the frame algorithm that achieves the same guaranteed rate of convergence as the classical frame algorithm without requiring knowledge of the frame bounds.  There are two primary motivations for this endeavor.  First, such an algorithm will provide a simpler alternative to the conjugate-gradient algorithm for applications in which estimates of the frame bounds are not readily available.  Second, a frame algorithm with these properties may be used to improve upon existing reconstruction algorithms associated with certain nonlinear frame inversion problems, e.g., the $\lambda$-saturated frame algorithm arising in the study of saturation recovery \cite{AFGJR2025}. 

The remainder of this note is organized as follows.  Section \ref{section:2} introduces notation and background information that will be used in subsequent sections.  Section \ref{section:3} includes a description of a greedy version of the frame algorithm along with one of the main results of this work, which establishes the convergence and error bound for the greedy frame algorithm.  In Section \ref{section:4}, the second main result, establishing the robustness of the greedy frame algorithm, is proven.  This theorem ensures that the greedy algorithm can be employed successfully with noisy measurements.  Finally, Section 5 presents numerical examples to illustrate the performance of the greedy frame algorithm in comparison to its non-adaptive counterparts.

\section{Notation and Preliminaries} \label{section:2}

Throughout this note, $\mathbb{H}$ will denote a separable, real or complex Hilbert space.  It will be convenient to establish notation for several related operators associated with a frame $\lbrace x_{j} \rbrace_{j\in J}$ of $\mathbb{H}$, where the index set $J$ is assumed to be countable.  The \emph{analysis operator} $T:\mathbb{H} \mapsto \ell^{2}(J)$ is defined by
$$ Tx = \lbrace \langle x,x_{j}\rangle \rbrace_{j\in J}$$

\noindent
and its adjoint, the \emph{synthesis operator}, $T^{*}:\ell^{2}(J) \mapsto \mathbb{H}$ is given by
$$ T^{*} c = \sum_{j\in J} c_{j} x_{j}.$$

\noindent
The \emph{frame operator} $S:\mathbb{H} \mapsto \mathbb{H}$ is the composition of the synthesis and analysis operators, i.e.,
$$ S x = T^{*} T x = \sum_{j\in J} \langle x,x_{j} \rangle x_{j}.$$

The frame operator associated with any frame must be invertible and, thus, the frame inversion problem is morally equivalent to the problem of inverting the frame operator.  At the heart of the frame algorithm is the fact that $\Vert I-\alpha S \Vert < 1$ for $0<\alpha < \frac{2}{B}$, which implies that $\alpha S$ and, consequently, $S$ are invertible operators.  In particular, this observation motivates an expression for $S^{-1}$ based on the Neumann series for $(\alpha S)^{-1}$, namely,
$$ S^{-1} = \alpha (\alpha S)^{-1} = \alpha \sum_{n=0}^{\infty} (I-\alpha S)^{n}.$$

\noindent
It should be no surprise that the approximations stemming from the frame algorithm can be related to partial sums of this series.

\begin{proposition*}
Let $S$ be the frame operator associated with a frame $\lbrace x_{j} \rbrace_{j\in J}$ for a separable Hilbert space $\mathbb{H}$.  Fix $0<\alpha < \frac{2}{B}$ and let $y_{0}=0$.  Define, for $n\ge 1$,
$$ F_{n} = \alpha \sum_{k=0}^{n-1} (I-\alpha S)^{k}  \qquad \text{and} \qquad y_{n+1} = y_{n} + \alpha S(x-y_{n}). $$

\noindent
Then, $y_{n} = F_{n} Sx$ for each integer $n\ge 1$.
\end{proposition*}

\begin{proof}
Observe that $F_{1} =\alpha I$ so $F_{1} Sx = \alpha Sx$, while $y_{1} = y_{0} + \alpha S(x-y_{0}) = \alpha Sx$.  Thus the claimed identity holds when $n=1$.  Assume that the identity holds for $1\le k\le n$.  Observe that 
$$ F_{n+1} S x = \left ( \alpha I + (I-\alpha S) F_{n} \right ) Sx = \alpha Sx + (I-\alpha S) y_{n} = y_{n} + \alpha S(x-y_{n}) = y_{n+1},$$

\noindent
which verifies the induction step and completes the proof.
\end{proof}

\section{A Greedy Frame Algorithm} \label{section:3}

This section develops an adaptive frame algorithm based on a recursion of the form
$$ y_{n+1} = y_{n} + \alpha_{n} \sum_{j\in J} \langle x-y_{n}, x_{j} \rangle x_{j}, $$

\noindent
in which $\alpha_{n}>0$ varies with $n$.  The relaxation parameter $\alpha_{n}$ will be chosen according to
$$ \alpha_{n} = \argmin_{\alpha > 0} \Vert x-y_{n+1} \Vert_{S}^{2}, $$

\noindent
where $\Vert \cdot\Vert_{S}$ is the equivalent norm on $\mathbb{H}$ associated with the inner product $ \langle x, y \rangle_{S} = \langle Sx, y \rangle$ for $x,y\in \mathbb{H}$.  This choice for $\alpha_{n}$ leads to the greatest improvement of the approximation in terms of $\Vert \cdot \Vert_{S}$ at each step.  The following elementary fact will simplify the proof of the main result in this section.

\begin{lemma*} 
Let $S$ be a positive definite operator on a Hilbert space $\mathbb{H}$ and define $\Vert \cdot \Vert_{S} = \sqrt{\langle S\cdot, \cdot \rangle}$.  If $p(S)$ is a polynomial in $S$, then $\Vert p(S)\Vert_{S} = \Vert p(S)\Vert$.
\end{lemma*}

\begin{proof}
It follows from the spectral theorem that there exists a positive definite operator $A$ such that $A^{2}=S$.  It will be convenient to denote $A$ by $S^{\frac{1}{2}}$.  If $p(S)$ is a polynomial in $S$, then
$$ \Vert p(S) \Vert_{S} = \sup_{x\neq 0} \frac{\vert \langle p(S) x,x\rangle_{S}\vert}{\Vert x \Vert_{S}} 
                        = \sup_{x\neq 0} \frac{\vert \langle p(S)S^{\frac{1}{2}}x,S^{\frac{1}{2}}x\rangle \vert}{\Vert S^{\frac{1}{2}} x \Vert} 
                        = \sup_{x\neq 0} \frac{\vert \langle p(S)x,x\rangle \vert}{\Vert x \Vert} 
                        = \Vert p(S)\Vert, $$

\noindent
where the second to last equality follows from the fact that $S$ and, hence, $S^{\frac{1}{2}}$ are invertible.
\end{proof}

With this setup, a greedy version of the frame algorithm will now be described.

\begin{theorem}[Greedy Frame Algorithm] \label{thm:greedy}
Let $\lbrace x_{j} \rbrace_{j\in J}$ be a frame for $\mathbb{H}$ with frame operator $S$.  Fix $x \in \mathbb{H}$.  Let $y_{0}=0$ and, for $n\ge 0$, define
\begin{equation} \label{eq:greedy}
\begin{aligned}
\alpha_{n} &= \frac{\Vert S(x-y_{n})\Vert^{2}}{\Vert S(x-y_{n})\Vert_{S}^{2}} \\[4pt]
y_{n+1} &= y_{n} + \alpha_{n} S(x-y_{n}).
\end{aligned}
\end{equation}

\noindent
Then, $y_{n}$ converges to $x$ as $n\rightarrow \infty$.  Moreover, if the optimal lower and upper frame bounds for $\lbrace x_{j} \rbrace_{j\in J}$ are $A$ and $B$, respectively, then
\begin{equation} \label{eq:greedy-bound}
\Vert x-y_{n} \Vert_{S} \le \big ( \tfrac{B-A}{B+A} \big )^{n} \Vert x \Vert_{S}.
\end{equation}
\end{theorem}

\begin{proof}
The error of approximation at step $n+1$ is given in the $\Vert \cdot\Vert_{S}$ norm by
$$ \begin{aligned} \Vert x-y_{n+1} \Vert_{S}^{2} &= \Vert (I-\alpha_{n}S)(x-y_{n})\Vert_{S}^{2} \\
                                                 &= \langle (x-y_{n})-\alpha_{n}S(x-y_{n}), (x-y_{n})-\alpha_{n}S(x-y_{n})\rangle_{S} \\
                                                 &= \Vert x-y_{n}\Vert_{S}^{2} - 2\alpha_{n} \langle S(x-y_{n}),x-y_{n} \rangle_{S} + \alpha_{n}^{2} \Vert S(x-y_{n})\Vert_{S}^{2}. \end{aligned} $$

\noindent
The approximation error is thus a quadratic function of $\alpha_{n}$ with a global minimum at
$$ \alpha_{n} = \frac{\langle S(x-y_{n}),x-y_{n}\rangle_{S}}{\Vert S(x-y_{n})\Vert_{S}^{2}} = \frac{\Vert S(x-y_{n})\Vert^{2}}{\Vert S(x-y_{n})\Vert_{S}^{2}}. $$

\noindent
The lemma implies that the operator norm of $I-\alpha S$, $\alpha \in \mathbb{R}$, is unaffected by the change in norm, i.e.,
$$ \Vert I-\alpha S\Vert_{S} = \Vert I-\alpha S\Vert. $$

\noindent
Let $\beta=2/(A+B)$ where $A$ and $B$, respectively, are the optimal lower and upper frame bounds.  It follows that
$$ \begin{aligned}
   \Vert x-y_{n+1}\Vert_{S} &= \Vert (I-\alpha_{n}S)(x-y_{n})\Vert_{S} \\[1pt]
                        &= \min_{\alpha \in \mathbb{R}} \Vert (I-\alpha S)(x-y_{n})\Vert_{S} \\[1pt]
                        &\le \Vert  ( I - \beta S ) (x-y_{n}) \Vert_{S} \\[1pt]
                        &\le \big ( \tfrac{B-A}{B+A} \big ) \Vert x-y_{n} \Vert_{S}. 
   \end{aligned} $$

\noindent
This estimate justifies both the claim that $y_{n}$ converges to $x$ as well as the error bound \eqref{eq:greedy-bound}.
\end{proof}

A key advantage of the greedy frame algorithm lies in the fact that it achieves the same guaranteed rate of convergence as the traditional frame algorithm without any knowledge of the frame bounds.  It is also possible to formulate a greedy frame algorithm based on the choice
$$ \alpha_{n} = \argmin_{\alpha >0} \Vert x-y_{n+1} \Vert^{2}. $$

\noindent
The proof of Theorem \ref{thm:greedy} is easily adapted for the change in norm, so this alternative form of the greedy frame algorithm will be presented without proof.  Note, however, that the expression for the optimal relaxation parameter differs from \eqref{eq:greedy} in this case.

\begin{proposition} \label{thm:greedyprop}
Let $\lbrace x_{j} \rbrace_{j\in J}$ be a frame for $\mathbb{H}$ with frame operator $S$.  Fix $x \in \mathbb{H}$.  Let $y_{0}=0$ and, for $n\ge 0$, define
\begin{equation} \label{eq:alpha}
\begin{aligned}
\alpha_{n} &= \frac{\Vert x-y_{n}\Vert_{S}^{2}}{\Vert S(x-y_{n})\Vert^{2}} \\[4pt]
y_{n+1} &= y_{n} + \alpha_{n} S(x-y_{n}).
\end{aligned}
\end{equation}

\noindent
Then, $y_{n}$ converges to $x$ as $n\rightarrow \infty$.  Moreover, if the optimal lower and upper frame bounds for $\lbrace x_{j} \rbrace_{j\in J}$ are $A$ and $B$, respectively, then
\begin{equation} \label{eq:errorbound}
\Vert x-y_{n} \Vert \le \big ( \tfrac{B-A}{B+A} \big )^{n} \Vert x \Vert.
\end{equation}
\end{proposition}

The motivation to emphasize the greedy frame algorithm of Theorem \ref{thm:greedy} over that of Proposition \ref{thm:greedyprop} relates to the ease of implementation with noisy measurements, which will be discussed in the next section.

\section{Robustness of the Greedy Frame Algorithm} \label{section:4}

It is important to show that the adaptive algorithm presented in Theorem \ref{thm:greedy} can be be implemented based on noisy measurements of the frame coefficients without substantial loss in performance.  In particular, it is essential that the optimal relaxation parameter at step $n$ be computable in terms of the noisy measurements.  The main result of this section, stated below, demonstrates the fact that the greedy frame algorithm of Theorem \ref{thm:greedy} is robust to measurement error.

\begin{theorem} \label{thm:robust}
Let $\lbrace x_{j} \rbrace_{j\in J}$ be a frame for $\mathbb{H}$ with analysis operator $T$ and optimal lower and upper frame bounds $A$ and $B$, respectively.  Let $S$ represent the frame operator, $T^{*}T$.  Fix $x\in \mathbb{H}$ and $\delta_{0} >0$.  Let $c\in \ell^{2}(J)$ such that $\Vert Tx-c\Vert_{2} \le \delta_{0}$, define $y_{0}=0$ and, for $n\ge 0$, 
\begin{equation} \label{eq:greedy-2}
\begin{aligned}
\alpha_{n} &= \frac{\Vert T^{*}(c-Ty_{n})\Vert^{2}}{\Vert T^{*}(c-Ty_{n})\Vert_{S}^{2}} \\[4pt]
y_{n+1} &= y_{n} + \alpha_{n} T^{*}(c-Ty_{n}).
\end{aligned}
\end{equation}

\noindent
Then, 
\begin{equation} \label{eq:bound-2}
\Vert x-y_{n} \Vert_{S} \le \left ( \frac{B-A}{B+A} \right )^{n} \left ( \Vert x \Vert_{S} + 2\delta_{0}\right ) +  2\delta_{0}, \qquad n\ge 1.
\end{equation}
\end{theorem}

\begin{proof}
Fix $x\in \mathbb{H}$ and $c\in \ell^{2}(J)$ such that $\Vert Tx-c\Vert_{2} \le \delta_{0}$.  Define $\delta \in \ell^{2}(J)$ such that $c = Tx + \delta$ and observe that $\Vert \delta\Vert \le \delta_{0}$.  Let $P$ represent the orthogonal projection onto the range of $T$ and define $\tilde{x} \in \mathbb{H}$ so that $T\tilde{x}=Pc$, which implies that $c = T\tilde{x} + \tilde{c}$ where $\tilde{c}$ belongs to the null space of $T{^*}$.  With this setup, it follows from the second part of \eqref{eq:greedy-2} that
$$ \begin{aligned} 
    \tilde{x} - y_{n+1} &= \tilde{x} - y_{n} - \alpha_{n} T^{*}(c-T y_{n}) \\
                        &= \tilde{x} - y_{n} - \alpha_{n} T^{*}(T\tilde{x} + \tilde{c} - T y_{n}) \\
                        &= \tilde{x} - y_{n} - \alpha_{n} S(\tilde{x}-y_{n}) \\
                        &= (I-\alpha_{n}S) (\tilde{x}-y_{n}). 
    \end{aligned} $$

\noindent
In light of the previous calculation, one can write
$$ \Vert \tilde{x}-y_{n+1} \Vert_{S}^{2} = \Vert \tilde{x}-y_{n} \Vert_{S}^{2} - 2\alpha_{n} \langle S(\tilde{x}-y_{n}),\tilde{x}-y_{n}\rangle_{S} + \alpha_{n}^{2} \Vert S(\tilde{x}-y_{n}) \Vert_{S}^{2}. $$

\noindent
This expression for the error of approximation is minimized when
$$ \alpha_{n} = \frac{\Vert S(\tilde{x}-y_{n})\Vert^{2}}{\Vert S(\tilde{x}-y_{n})\Vert_{S}^{2}} = \frac{\Vert T^{*}(c-Ty_{n})\Vert^{2}}{\Vert T^{*}(c-Ty_{n})\Vert_{S}^{2}}$$

\noindent
and, repeating the argument of Theorem \ref{thm:greedy}, the error of approximation must satisfy
$$ \Vert \tilde{x}-y_{n+1} \Vert_{S} \le \left ( \frac{B-A}{B+A} \right ) \Vert \tilde{x}-y_{n} \Vert_{S}.$$

\noindent
It follows that
$$ \Vert \tilde{x}-y_{n} \Vert_{S} \le \left ( \frac{B-A}{B+A} \right )^{n} \Vert \tilde{x}-y_{0} \Vert_{S} = \left ( \frac{B-A}{B+A} \right )^{n} \Vert \tilde{x} \Vert_{S}, \qquad n \ge 1.$$

\noindent
It remains to relate this bound to the original vector $x$.  Since $T\tilde{x}=Pc$ is the orthogonal projection of $c$ onto the range of $T$, it follows that
$$ \Vert T\tilde{x} - c \Vert \le \Vert Tx - c \Vert \le \delta_{0}. $$

\noindent
In order to bound $\Vert x-\tilde{x}\Vert_{S}$, one can write
$$ \Vert x-\tilde{x} \Vert_{S}^{2} = \langle  S(x-\tilde{x}),  x-\tilde{x} \rangle = \Vert T(x-\tilde{x})\Vert^{2}, $$

\noindent
which shows that 
$$ \Vert x-\tilde{x}\Vert_{S} = \Vert T(x-\tilde{x})\Vert = \Vert (Tx-c)-(T\tilde{x}-c)\Vert \le 2\delta_{0}. $$ 

\noindent
Finally, combining this with the fact that $\Vert x-y_{n} \Vert_{S} \le \Vert \tilde{x}-y_{n}\Vert_{S} + \Vert x-\tilde{x}\Vert_{S}$ due to the triangle inequality, it follows that
$$ \Vert x-y_{n}\Vert_{S} \le \left ( \frac{B-A}{B+A} \right )^{n} \Vert \tilde{x} \Vert_{S} + 2 \delta_{0} \le \left ( \frac{B-A}{B+A} \right )^{n} \left ( \Vert x \Vert_{S} + 2\delta_{0} \right ) + 2 \delta_{0}, $$

\noindent
completing the proof of \eqref{eq:bound-2}.
\end{proof}

\begin{remark*}
Notice that the optimal parameter $\alpha_{n}$ described by \eqref{eq:greedy-2} is computable in terms of the sequence of noisy measurements, $c$.  This is not guaranteed for the adaptive frame algorithm described in Proposition \ref{thm:greedyprop}.  Adopting the notation used in the proof of Theorem \ref{thm:robust}, the error estimate in the standard norm becomes
$$ \Vert \tilde{x}-y_{n+1} \Vert^{2} = \Vert \tilde{x}-y_{n} \Vert^{2} - 2\alpha_{n} \langle S(\tilde{x}-y_{n}),\tilde{x}-y_{n}\rangle + \alpha_{n}^{2} \Vert S(\tilde{x}-y_{n}) \Vert^{2}, $$

\noindent
which is minimized when 
$$ \alpha_{n} = \frac{\Vert \tilde{x}-y_{n}\Vert_{S}^{2}}{\Vert S(\tilde{x}-y_{n})\Vert^{2}}. $$

\noindent
The expression in the denominator is easily computed in terms of the noisy measurements via
$$ \Vert S(\tilde{x}-y_{n})\Vert^{2} = \Vert T^{*}(T\tilde{x}-Ty_{n})\Vert^{2} = \Vert T^{*}(c-Ty_{n})\Vert^{2}, $$

\noindent
however, the expression in the numerator relies on knowledge of $\tilde{c}$, due to the fact that
$$ \Vert \tilde{x}-y_{n}\Vert_{S}^{2} = \Vert T(\tilde{x}-y_{n}) \Vert^{2} = \Vert c-\tilde{c}-Ty_{n}\Vert^{2} = \Vert c-Ty_{n}\Vert^{2} - \Vert \tilde{c} \Vert^{2}. $$

\noindent
Here, the last simplification makes use of the fact that $\langle c,\tilde{c}\rangle = \Vert \tilde{c}\Vert^{2}$.  Thus the fact that $\tilde{c}$ cannot be readily computed presents an obstacle in the selection of the optimal relaxation parameter when the algorithm of Proposition \ref{thm:greedyprop} is implemented with noisy measurements.
\end{remark*}

\section{Numerical Examples} \label{section:5}

Two numerical examples will be used to illustrate the potential benefits of the greedy frame algorithm described in Theorems \ref{thm:greedy} and \ref{thm:robust}.  The underlying frame for each example will be a Parseval frame ($A=B=1$) resulting from a random projection of an orthonormal discrete cosine basis of $\mathbb{R}^{N}$ into $\mathbb{R}^{d}$, where $N>d$.  The discrete cosine basis $\lbrace e_{j}\rbrace_{j=0}^{N-1}$ will be defined by
$$ e_{j}(k) = \begin{cases}  \sqrt{\tfrac{1}{N}}, & j=0 \\ \sqrt{\tfrac{2}{N}} \cos{\left ( \tfrac{\pi}{N} (j+\tfrac{1}{2}) k\right )}, & 1\le j\le N-1, \end{cases}  \qquad 0\le k\le N-1. $$

\noindent
The Parseval frame is obtained by projection onto $d$ randomly selected coordinates, chosen from a uniform distribution.  Each experiment will summarize the results of 1000 trials in terms of the mean error of approximation over the first 50 iterations of the corresponding reconstruction algorithms for a unit vector in $\mathbb{R}^{d}$.  The unit vector will be chosen from a uniform distribution and both a random frame and unit vector are selected in each trial.  Shaded error bars will indicate the 10th and 90th percentiles for the error of approximation over the 1000 trials.

\begin{example} \label{eg1}
The main advantage of the greedy frame algorithm lies in the fact that it does not require knowledge of the frame bounds.  To illustrate the utility of this trait, consider a frame inversion problem in which highly redundant, noisy measurements are exposed to a number of erasures.  The remaining frame coefficients are no longer associated with a Parseval frame and the lower frame bound for the corresponding frame vectors is unknown.  For this reason, the classical frame algorithm will be implemented using the relaxation constant $\alpha=1$ based on the fact that the upper frame bound is at most 1.  Fix $d=100$ and $N=200$.  The noisy frame measurements $c\in \mathbb{R}^{N}$ are determined by $c=Tx+e$, where the coordinates of $e$ are chosen according to the standard normal distribution and then renormalized so that $\Vert e\Vert_{2} = 10^{-6}$.  Following the addition of noise, 10 frame coefficients are randomly erased.  The greedy frame algorithm will implement an adaptive relaxation parameter computed according to \eqref{eq:greedy-2}, where $S$ now represents the frame operator associated with the frame vectors remaining after the erasures.  The results of this experiment are shown in Figure \ref{fig:eg1}.  Between Iteration 1 and Iteration 15, the mean rate of error reduction was approximately 0.64 for the classical frame algorithm and 0.47 for the greedy frame algorithm.  The error of approximation for the greedy algorithm bottoms out after roughly 20 iterations, while the classical algorithm requires a little less than twice that number.

 \begin{figure}[ht]
 \centering
 \includegraphics[width=0.625\textwidth]{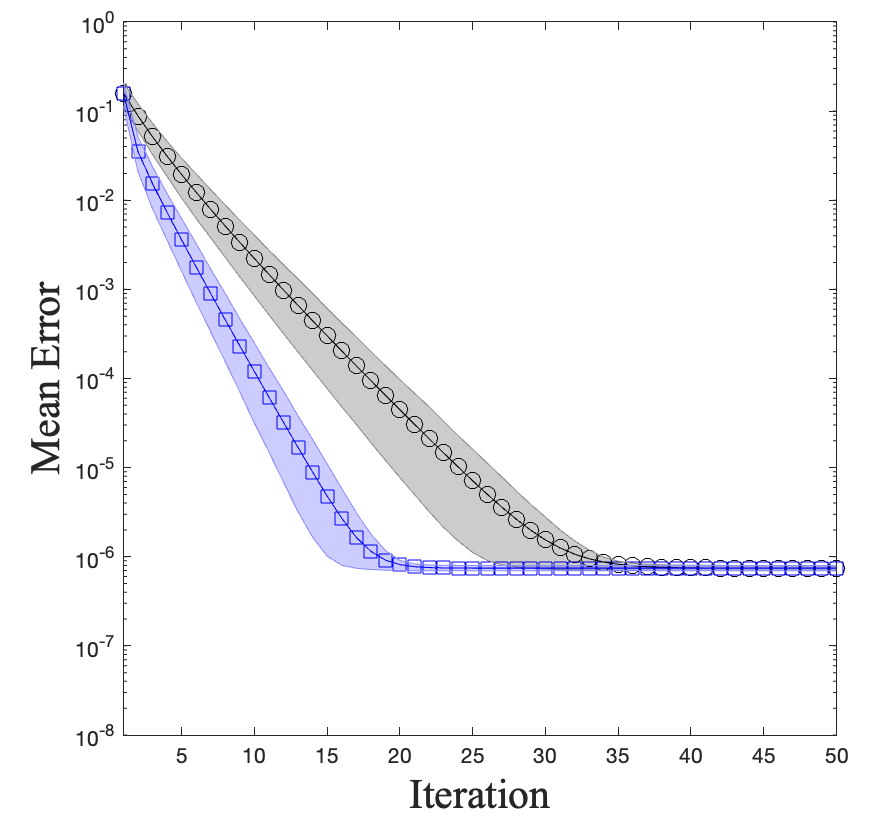}
 \caption{Error of approximation for the classical frame algorithm (black) and the greedy frame algorithm (blue) subject to both erasures and noisy measurements.} \label{fig:eg1}
 \end{figure}
\end{example}

\begin{example} \label{eg2}
Saturation recovery (see \cite{AFGJR2025}) provides another context in which the greedy algorithm can be leveraged to an advantage.  In real-valued saturation recovery, the goal is to reconstruct a vector $x$ from its saturated frame coefficients, $\lbrace \phi_{\lambda}(\langle x,x_{j}\rangle ) \rbrace_{j\in J}$, where $\lambda >0$ and 
$$ \phi_{\lambda}(t) = \begin{cases} -\lambda, & t\le -\lambda \\ t, & -\lambda < t < \lambda, \\ \lambda, & t\ge \lambda. \end{cases} $$

\noindent
It is possible to reconstruct $x$ from its saturated frame coefficients provided that the frame vectors associated with the unsaturated coordinates form a frame \cite[Theorem 5.2]{AFGJR2025}.  This is accomplished via the so called \emph{$\lambda$-saturated frame algorithm}, which utilizes the recursive formula
$$ y_{k+1} = y_{k} + \frac{2}{A+B} \sum_{j\in J_{\lambda}^{\sharp}(x)\cup J_{S}(x,y_{k})} \left ( \phi_{\lambda}(\langle x,x_{j}\rangle) - \langle y,x_{j} \rangle \right ) x_{j}, $$

\noindent
where $A$ and $B$ are upper and lower frame bounds for $\lbrace x_{j} \rbrace_{j\in J}$ and the index sets $J_{\lambda}^{\sharp}(x)$ and $J_{S}(x,y)$ are defined by
$$ \begin{aligned}
  J_{\lambda}^{\sharp}(x) &= \left \lbrace j\in J : \vert \langle x, x_{j} \rangle \vert <\lambda \right \rbrace \\
  J_{S}(x,y) & = \left \lbrace j\in J : \langle x,x_{j} \rangle \le -\lambda < \langle y,x_{j}\rangle  \text{ or } \langle y, x_{j}\rangle < \lambda \le \langle x,x_{j} \rangle \right \rbrace. \end{aligned} $$

\noindent
In this example, the underlying frame is Parseval, so $A=B=1$ and $\tfrac{2}{A+B}=1$.  An adaptive version of the $\lambda$-saturated frame algorithm is obtained by choosing the relaxation parameter $\alpha_{k}$ based on \eqref{eq:greedy} with the understanding that $S$ corresponds to the frame operator of $\lbrace x_{j} \rbrace_{j\in J_{\lambda}^{\sharp}(x)\cup J_{S}(x,y_{k})}$.  Fix $d=100$, $N=250$, and $\lambda=0.08$.  The results of this experiment are shown in Figure \ref{fig:eg2}.  Between Iteration 1 and Iteration 50, the mean rate of error reduction was approximately 0.85 for the $\lambda$-saturated frame algorithm, but improved to 0.76 with the greedy implementation.  

 \begin{figure}[ht]
 \centering
 \includegraphics[width=0.625\textwidth]{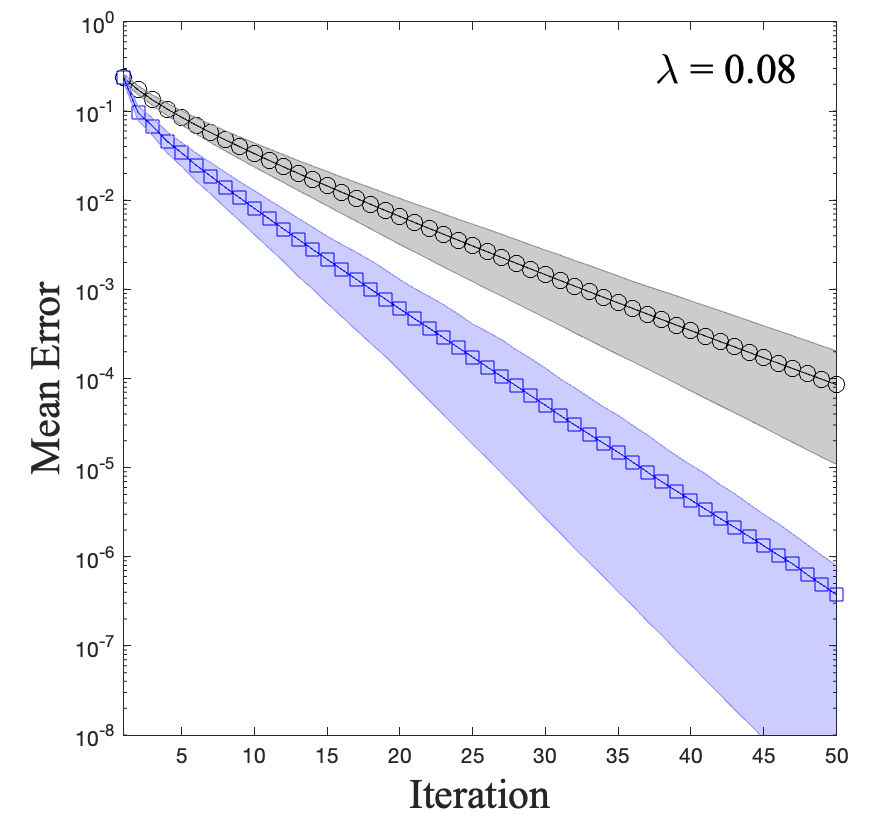}
 \caption{Error of approximation for the $\lambda$-saturated frame algorithm (black) and its greedy counterpart (blue) subject to saturation level $\lambda = 0.08$.} \label{fig:eg2}
 \end{figure}
\end{example}

\end{document}